\documentclass[11pt]{amsart}

\usepackage[utf8]{inputenc}
\usepackage[T1]{fontenc}
\usepackage{amsmath, amssymb, amsthm}
\usepackage{mathrsfs}
\usepackage{geometry}
\usepackage{enumitem}
\usepackage{booktabs}
\usepackage{array}
\usepackage{graphicx}
\usepackage{float}
\usepackage{hyperref}

\hypersetup{
    colorlinks=true,
    linkcolor=blue,
    citecolor=blue,
    urlcolor=blue,
    pdftitle={Casorati Inequalities for Pointwise Slant Riemannian Submersions},
    pdfauthor={Tanveer Fatima}
}

\theoremstyle{plain}
\newtheorem{theorem}{Theorem}[section]
\newtheorem{corollary}{Corollary}[section]
\newtheorem{proposition}{Proposition}[section]
\newtheorem{lemma}[proposition]{Lemma}

\theoremstyle{definition}
\newtheorem{definition}[proposition]{Definition}
\newtheorem{example}{Example}[section]

\theoremstyle{remark}
\newtheorem{remark}[proposition]{Remark}

\DeclareMathOperator{\Span}{span}

\newcommand{\R}{\mathbb{R}}
\newcommand{\C}{\mathbb{C}}

\newcommand{\calH}{\mathcal{H}}
\newcommand{\calV}{\mathcal{V}}

\title[Casorati Inequalities for Pointwise Slant Submersions]{Casorati Inequalities along Vertical and Horizontal Distributions of Pointwise Slant Riemannian Submersions}
\author{Tanveer Fatima$^{1,2}$}
\address{$^{1}$Department of Mathematics and Statistics, College of Science in Yanbu, Taibah University, Yanbu Governorate, Saudi Arabia}
\address{$^{2}$ Health and Life Research Center, Taibah University, Madinah, Saudi Arabia}
\email{tansari@taibahu.edu.sa}

\author{Sana Abdulkarem Alharbi$^{1,2,*}$}
\address{$^{1}$Department of Mathematics and Statistics, College of Science in Yanbu, Taibah University, Yanbu Governorate, Saudi Arabia}
\address{$^{2}$ Health and Life Research Center, Taibah University, Madinah, Saudi Arabia}
\email{$^{*}$saaharbi@taibahu.edu.sa}
\thanks{$^{*}$Corresponding author.}

\author{Sharief Deshmukh$^{3}$}
\address{$^{3}$Department of Mathematics, College of Science, King Saud University, P.O. Box 2455, Riyadh 11451, Saudi Arabia}
\email{shariefd@ksu.edu.sa}

\subjclass[2020]{53C20, 53C25, 53C40, 53C42, 53B25}
\keywords{Riemannian submersions, O'Neill tensors, space forms, Casorati curvatures, pointwise slant submersions, optimal inequalities}

\begin{document}

\begin{abstract}
We establish optimal Casorati inequalities for pointwise slant Riemannian
submersions. As such a submersion carries two distinct distributions, the
vertical, tangent to the fibers, and the horizontal, governed by O'Neill's
tensors $T$ and $A$, we treat them separately. For submersions from generalized
complex and generalized Sasakian space forms we bound the normalized scalar
curvature of each distribution by its normalized Casorati curvatures. Special cases recover sharp
inequalities for real, complex, K\"ahler, Sasakian, Kenmotsu, cosymplectic, and
almost $C(\alpha)$ space forms. Equality is characterized geometrically:
invariantly quasi-umbilical fibers in the vertical case and integrability
($A\equiv 0$) in the horizontal case. The invariant and anti-invariant limits
reproduce known results, and examples illustrate sharpness.
\end{abstract}

\maketitle

\section{Introduction}
A recurring difficulty with Gaussian curvature is that it may be zero even where
a surface is visibly curved, so it does not always reflect the shape of a
surface as it sits in space. It was precisely this gap that led Casorati
\cite{Casorati1890} in 1890 to propose a different extrinsic invariant, today
known as the Casorati curvature. Unlike the Gaussian curvature, it vanishes
exactly at planar points, which makes it a natural tool for shape analysis and
visualization \cite{Decu2008, Ons2011} and places it within the broader study of
curvature invariants that has shaped differential geometry since the work of
Gauss and Riemann.

Interest in Casorati inequalities in their present form was sparked by Decu,
Haesen, and Verstraelen \cite{Decu2008}, who building on Chen's program of
$\delta$-invariants \cite{Chen1990, Chen2011} derived sharp estimates bounding
the scalar curvature of a submanifold of a real space form by its normalized
Casorati curvatures. Their work prompted an extensive line of research, with sharp Casorati
inequalities established across a wide range of ambient spaces and submanifold
classes. Among these contributions, Aquib et al.\ \cite{Aquib2019, Aquib2018}
treated Lagrangian and bi-slant submanifolds of complex space forms; Lee et al.\
\cite{Lee2020a, Lee2022, Lee2017} addressed Legendrian submanifolds in Sasakian
and Kenmotsu space forms; Lone \cite{Lone2017a, Lone2017b, Lone2019, Lone2019a}
examined slant submanifolds of generalized complex and Sasakian space forms;
Vilcu \cite{Vilcu2018} considered Lagrangian submanifolds of complex space
forms; Zhang et al.\ \cite{Zhang2016, Zhang2016a} dealt with submanifolds of
real and complex space forms; and Siddiqui \cite{Siddiqui2018} studied bi-slant
submanifolds of generalized Sasakian space forms. A broader overview can be
found in the survey of Chen \cite{Chen2021}.

Parallel to submanifold theory, Riemannian submersions emerged as important
objects with applications in theoretical physics, particularly in Kaluza--Klein
theory and supergravity \cite{Falcitelli2004, ONiell1966}. O'Neill
\cite{ONiell1966} developed the fundamental equations of a submersion,
introducing two tensor fields $T$ and $A$: the tensor $T$ measures the failure
of the fibers to be totally geodesic, while $A$ measures the failure of the
horizontal distribution to be integrable. This is the central structural
feature that distinguishes submersions from submanifolds: a submersion carries
{two} natural distributions with genuinely different geometric meanings.
The vertical distribution $\calV=\ker F_*$ is tangent to the fibers and its
extrinsic geometry is encoded by $T$, whereas the horizontal distribution
$\calH=(\ker F_*)^{\perp}$ is governed by $A$. In
\cite{Sahin2010a, Sahin2013, Sahin2017}, Sahin developed a comprehensive theory
of submersions in Hermitian geometry, and these structures are special cases of
the Riemannian maps introduced by Fischer \cite{Fischer1992}.

The Casorati geometry of these two distributions must therefore be analyzed
separately, and the resulting equality conditions differ: invariant
quasi-umbilicity of the fibers on the vertical side, and integrability of the
horizontal distribution on the horizontal side. This two-distribution
phenomenon is precisely what separates the submersion problem from the Casorati
theory of Riemannian maps. For a Riemannian map the differential need not be
surjective, and the extrinsic geometry of the map is carried by a single
object, its second fundamental form $\nabla F_*$, which takes values in the
orthogonal complement $(\operatorname{range} F_*)^{\perp}$ of the range inside
the target; the corresponding Casorati inequality bounds this one tensor. A
Riemannian submersion is exactly the case in which $F_*$ is onto, so
$(\operatorname{range} F_*)^{\perp}=0$ and that tensor vanishes identically. The
extrinsic geometry then passes entirely to the source, where it is shared
between the two O'Neill tensors acting on orthogonal distributions, $T$, the
second fundamental form of the fibers, on $\calV$, and $A$, the obstruction to
horizontal integrability, on $\calH$. Since $T$ and $A$ encode independent
curvature data through O'Neill's equations and act on different distributions,
there is no single ambient form to optimize; each distribution instead yields
its own optimal Casorati inequality, which is the structural reason the
submersion setting requires two separate families of results.

Casorati inequalities were subsequently extended to Riemannian maps and
submersions. Lee et al.\ \cite{Lee2021} obtained optimal inequalities in real
and complex space forms, Polat et al.\ \cite{Polat2025} extended these to
Sasakian space forms, and Zaidi and Shanker \cite{Zaidi2024} treated bi-slant
maps to Kenmotsu manifolds, while Fatima et al.\ \cite{Fatima2025} investigated
nearly K\"ahler targets. However, all of these results treated only
{invariant} ($\theta=0$) or {anti-invariant} ($\theta=\pi/2$)
structures. Our results contain these as limiting cases. Set $\theta=0$, Theorems~\ref{thm:complex_subm_vert}
and~\ref{thm:sasakian_subm_vert} recover the invariant submersion inequalities
of \cite[Theorem~4.1]{Lee2021} and \cite[Theorem~5.1]{Polat2025} (as corrected
in~\cite{Polat2025a}), and Theorem~\ref{thm:complex_subm_horiz} recovers
\cite[Theorem~4.2]{Lee2021}; set $\theta=\pi/2$ eliminates the
$c_2$-terms and yields the corresponding anti-invariant submersion inequalities of
\cite{Lee2021, Polat2025} and the special case of \cite{Zaidi2024}. The full
slant parameter $\theta$ interpolates between these extremes.

The notion of slant submanifolds was introduced by Chen \cite{Chen1990} as a
natural generalization of invariant and anti-invariant submanifolds and was
extended to contact geometry \cite{Yano1984, Blair2010}. Etayo \cite{Etayo1998}
introduced {pointwise} slant submanifolds, where the angle varies from
point to point. For Riemannian maps, pointwise slant structures were studied by
Park \cite{Park2015}, G\"und\"uzalp and Akyol \cite{GunduzalpAkyol2022}, and, in
the horizontal Reeb case, by Demir and Yildirim \cite{DemirYildirim2025}. For
submersions, the pointwise slant theory was developed by Lee and Sahin
\cite{LeeSahin2014} from almost Hermitian manifolds and by Kumar--Prasad
\cite{Kumar2020} and Prasad--Kumar \cite{Prasad2021} from almost contact
manifolds. These structures unify the invariant, anti-invariant, and all
intermediate cases through a single varying slant angle. The parallel theory for pointwise slant Riemannian maps is developed in the
companion paper~\cite{FatimaMaps}.

Despite the extensive development of (a) Casorati inequalities for
invariant/anti-invariant submersions and (b) the pointwise slant theory of
submersions, Casorati inequalities for the {pointwise slant} case remain
unestablished. This gap is significant: geometrically, pointwise slant is the
general case, with the invariant and anti-invariant settings as limits;
physically, theories on the total space may involve spatially varying angles;
mathematically, a variable $\theta$ introduces new technical challenges that
require a careful, distribution-by-distribution optimization.

We address this gap by establishing comprehensive Casorati inequalities for
pointwise slant Riemannian submersions. Our contributions are as follows. First,
working separately on the vertical and horizontal distributions, we derive
explicit inequalities when the total space is a generalized complex space form
(Theorems~\ref{thm:complex_subm_vert} and~\ref{thm:complex_subm_horiz}) and a
generalized Sasakian space form (Theorems~\ref{thm:sasakian_subm_vert}
and~\ref{thm:sasakian_subm_horiz}), covering all classical space forms
(Corollaries~\ref{cor:subm_vert_real}, \ref{cor:subm_vert_sasakian},
\ref{cor:subm_horiz_real}, and~\ref{cor:subm_horiz_sasakian}). Second, we
characterize the equality cases: invariantly quasi-umbilical fibers
(Definition~\ref{def:quasi-umbilical}, Remark~\ref{rem:quasi-umbilical}) for the
vertical inequalities, and integrability of the horizontal distribution
(equivalently $A\equiv 0$) for the horizontal inequalities. Third, we verify
the reduction to the known results \cite{Lee2021, Polat2025, Zaidi2024} when
$\theta=0$ or $\theta=\pi/2$. Finally, we provide examples
(Examples~\ref{ex:subm1}--\ref{ex:subm3}) demonstrating sharpness and the role
of varying slant function.

The paper is organized as follows. Section~\ref{sec:preliminaries} collects the
preliminaries on Riemannian submersions, O'Neill's tensors $T$ and $A$, the
relevant space forms, and the algebraic optimization Lemma~\ref{lem:tripathi}.
Section~\ref{sec:pointwise_slant} develops pointwise slant structures for
submersions from almost Hermitian and from almost contact metric manifolds, and
proves the unifying norm identity (Proposition~\ref{prop:key_slant_identities}).
Section~\ref{sec:vertical} establishes the Casorati inequalities for the
vertical distribution, and Section~\ref{sec:horizontal} those for the horizontal
distribution, each for generalized complex and generalized Sasakian space forms.
Section~\ref{sec:examples} provides illustrative examples, and
Section~\ref{sec:discussion} offering concluding remark.

\section{Preliminaries}\label{sec:preliminaries}

In this section we establish the fundamental concepts needed for our main
theorems. Throughout, all manifolds are smooth, connected, and without
boundary.

\subsection{Riemannian Submersions}

Let $F:(M_1^{m_1},g_1)\to(M_2^{m_2},g_2)$ be a smooth submersion between
Riemannian manifolds. For $p\in M_1$, the kernel $\ker F_{*p}$ of the
differential and its orthogonal complement $(\ker F_{*p})^{\perp}$ give the
orthogonal decomposition $T_pM_1=\ker F_{*p}\oplus(\ker F_{*p})^{\perp}$. The
map $F$ is a {Riemannian submersion} if $F_{*p}:(\ker F_{*p})^{\perp}\to
T_{F(p)}M_2$ is a linear isometry for every $p\in M_1$ \cite{ONiell1966,
Falcitelli2004}.

For a Riemannian submersion $F : M_1 \to M_2$, we denote the 
vertical and horizontal distributions by 
$\calV = \ker F_*$ and $\calH = (\ker F_*)^\perp$, 
respectively. O'Neill \cite{ONiell1966} introduced two 
fundamental tensor fields $T$ and $A$, defined for vector 
fields $E, G$ on $M_1$ by
\begin{align}
T_E G &= \calH\nabla_{\calV E}\calV G 
        + \calV\nabla_{\calV E}\calH G, 
        \label{eq:tensor_T}\\
A_E G &= \calV\nabla_{\calH E}\calH G 
        + \calH\nabla_{\calH E}\calV G. 
        \label{eq:tensor_A}
\end{align}
Here $\calV$ and $\calH$ denote the vertical and horizontal 
projections, respectively. The tensor $T$ measures the 
obstruction to the fibers being totally geodesic, while $A$ 
measures the obstruction to the horizontal distribution being 
integrable.

\begin{proposition}[\cite{ONiell1966}]
\label{prop:oneill_tensors}
For a Riemannian submersion, the tensors $T$ and $A$ satisfy:
\begin{enumerate}[label=(\roman*)]
\item $T_U V$ is symmetric for $U,V \in \Gamma(\calV)$ and 
      $T_U W \in \Gamma(\calH)$ for $U \in \Gamma(\calV)$, 
      $W \in \Gamma(\calH)$.
\item $A_X Y$ is skew-symmetric for $X,Y \in \Gamma(\calH)$ 
      and $A_X U \in \Gamma(\calV)$ for $X \in \Gamma(\calH)$, 
      $U \in \Gamma(\calV)$.
\item For $U,V \in \Gamma(\calV)$: 
      $T_U V = \dfrac{1}{2}\calV[U,V]$.
\item For $X,Y \in \Gamma(\calH)$: 
      $A_X Y = \dfrac{1}{2}\calV[X,Y]$.
\end{enumerate}
\end{proposition}

The curvature properties of Riemannian submersions are 
described by the following Gauss-type equations.

\begin{theorem}[\cite{Falcitelli2004}]\label{thm:gauss_vertical}
For a Riemannian submersion $F: M_1 \to M_2$ and for 
$U,V,W,Z \in \Gamma(\calV)$,
\begin{align}\label{eq:gauss_vertical}
g_1(R^{M_1}(U,V)W,Z) = & g_1(\hat{R}(U,V)W,Z)- g_1(T_U W, T_V Z) \nonumber\\
& + g_1(T_V W, T_U Z),
\end{align}
where $\hat{R}$ denotes the curvature tensor of the fibers.
\end{theorem}

\begin{theorem}[\cite{Falcitelli2004}]
\label{thm:gauss_horizontal}
For a Riemannian submersion and for 
$X,Y,Z,W \in \Gamma(\calH)$,
\begin{align}\label{eq:gauss_horizontal}
g_1(R^{M_1}(X,Y)Z,W) = &g_1(R^{\calH}(X,Y)Z,W)- 2g_1(A_X Y, A_Z W) \nonumber\\
& + g_1(A_Y Z, A_X W) 
 - g_1(A_X Z, A_Y W),
\end{align}
where $R^{\calH}$ is the curvature tensor of the horizontal 
distribution.
\end{theorem}

\subsection{Space Forms}

We now recall the ambient spaces used throughout the paper.

\begin{definition}[\cite{Olszak1989, Tricerri1981}]
\label{def:gen_complex_space_form}
An almost Hermitian manifold $(M,g,J)$ is called a 
{generalized complex space form}, denoted $M(c_1,c_2)$, 
if its curvature tensor satisfies
\begin{align}\label{eq:gen_complex_curvature}
R(X,Y)Z &= c_1\{g(Y,Z)X - g(X,Z)Y\} \nonumber\\
&\quad + c_2\{g(X,JZ)JY - g(Y,JZ)JX 
            + 2g(X,JY)JZ\}
\end{align}
for smooth functions $c_1, c_2$ on $M$.
\end{definition}

Important special cases include:
\begin{enumerate}[label=(\roman*)]
\item {Real space forms}: $c_1 = c$, $c_2 = 0$ 
      (constant sectional curvature $c$).
\item {Complex space forms}: $c_1 = c_2 = c/4$ 
      (K\"{a}hler, constant holomorphic sectional curvature $c$).
\item {Real K\"{a}hler space forms}: 
      $c_1 = (c+3\alpha)/4$, $c_2 = (c-\alpha)/4$.
\end{enumerate}

Recall that an almost contact metric manifold $(M,\phi,\xi,\eta,g)$
carries a $(1,1)$-tensor field $\phi$, a Reeb vector field $\xi$, and a
$1$-form $\eta$ satisfying $\phi^2=-I+\eta\otimes\xi$, $\eta(\xi)=1$, and
$\phi\xi=0$ \cite{Blair2010,Yano1984}. In particular $\phi^2=-I$ on $\ker\eta=\{X\in TM:\eta(X)=0\}$, where
$\phi$ acts as an isometry.

\begin{definition}[\cite{Alegre2004, Blair2010}]
\label{def:gen_sasakian_space_form}
An almost contact metric manifold $(M,\phi,\xi,\eta,g)$ is 
called a {generalized Sasakian space form}, denoted 
$M(c_1,c_2,c_3)$, if
\begin{align}\label{eq:gen_sasakian_curvature}
R(X,Y)Z &= c_1\{g(Y,Z)X - g(X,Z)Y\} \nonumber\\
&\quad + c_2\{g(X,\phi Z)\phi Y - g(Y,\phi Z)\phi X 
            + 2g(X,\phi Y)\phi Z\} \nonumber\\
&\quad + c_3\{\eta(X)\eta(Z)Y - \eta(Y)\eta(Z)X 
            + g(X,Z)\eta(Y)\xi - g(Y,Z)\eta(X)\xi\}
\end{align}
for smooth functions $c_1, c_2, c_3$ on $M$.
\end{definition}

Special cases include:
\begin{enumerate}[label=(\roman*)]
\item {Sasakian space forms}: 
      $c_1 = (c+3)/4$, $c_2 = c_3 = (c-1)/4$.
\item {Kenmotsu space forms}: 
      $c_1 = (c-3)/4$, $c_2 = c_3 = (c+1)/4$.
\item {Cosymplectic space forms}: 
      $c_1 = c_2 = c_3 = c/4$.
\item {Almost $C(\alpha)$ space forms}: 
      $c_1 = (c+3\alpha^2)/4$, 
      $c_2 = c_3 = (c-\alpha^2)/4$.
\end{enumerate}

\subsection{A Key Algebraic Lemma}

The following lemma, due to Tripathi \cite{Tripathi2017}, 
provides the optimization result central to our proofs. It 
identifies the constrained minimum of a specific quadratic 
form on a hyperplane in $\mathbb{R}^r$.

\begin{lemma}[\cite{Tripathi2017}]\label{lem:tripathi}
Let $\Lambda = \{(z_1,\ldots,z_r) \in \mathbb{R}^r : 
z_1 + \cdots + z_r = k\}$ be a hyperplane in $\mathbb{R}^r$, 
and let $f: \mathbb{R}^r \to \mathbb{R}$ be the quadratic form
\begin{equation}\label{eq:quadratic_form_lemma}
f(z_1,\ldots,z_r) = \lambda_1\sum_{i=1}^{r-1}z_i^2 
+ \lambda_2 z_r^2 
- 2\sum_{1\le i<j\le r}z_i z_j,
\end{equation}
where $\lambda_1, \lambda_2 > 0$ satisfy 
$\lambda_2 = (r-1)/(\lambda_1 - r + 2)$. Then the constrained 
minimum of $f$ on $\Lambda$ is achieved at
\begin{equation}\label{eq:critical_point_lemma}
z_1 = \cdots = z_{r-1} = \dfrac{k}{\lambda_1+1}, \quad 
z_r = \dfrac{k(\lambda_1 - r + 2)}{\lambda_1+1},
\end{equation}
and at this critical point $f(z_1,\ldots,z_r) = 0$.
\end{lemma}

\begin{remark}\label{rem:tripathi_positivity}
The relation $\lambda_2 = (r-1)/(\lambda_1 - r + 2)$ requires 
$\lambda_1 > r - 2$ for the positivity condition $\lambda_2 > 0$ 
to hold. In the applications to Casorati inequalities throughout 
this paper, Lemma~\ref{lem:tripathi} is invoked with parameters 
$\lambda_1 = r$ and $\lambda_2 = (r-1)/2$ (see the proofs of Theorems~\ref{thm:complex_subm_vert} and~\ref{thm:complex_subm_horiz}). Since $r \geq 3$ by 
assumption, the positivity condition $\lambda_1 = r > r - 2$ 
is automatically satisfied. With $\lambda_1 = r$, the constraint 
$\lambda_2 = (r-1)/(\lambda_1 - r + 2)$ gives 
$\lambda_2 = (r-1)/2$ directly, confirming that 
the parameter choices are valid. The critical 
point~\eqref{eq:critical_point_lemma} then becomes
\[
z_1 = \cdots = z_{r-1} = \dfrac{k}{r+1}, \quad 
z_r = \dfrac{2k}{r+1},
\]
which will correspond to the quasi-umbilical ratio 
$h_{11}^\alpha = \cdots = h_{r-1,r-1}^\alpha 
= \dfrac{1}{2}\,h_{rr}^\alpha$ in the equality conditions 
of our main theorems.
\end{remark}

\begin{remark}\label{rem:rank_two}
The assumption $r \geq 3$ is necessary because for $r = 2$ 
the normalized Casorati curvatures $\delta_C(r-1) 
= \delta_C(1)$ vanish trivially, rendering the 
inequalities uninformative. Avoiding this degenerate 
case, we assume $r \geq 3$ throughout.
\end{remark}
\begin{remark}
In the companion article for Riemannian-map ~\cite{FatimaMaps}, a proper pointwise
slant structure forces $\operatorname{Ran}F_*$ to be even-dimensional, since
$P^{R}/\cos\theta$ is a complex structure on the range; the minimal admissible
rank is therefore $r=4$. No such parity constraint arises here:
$(P^{\calV})^2=-\cos^2\theta\,I$ holds only on the $\phi_1$-active part of the
distribution, which need not be the space over which the Casorati optimization
is performed. Hence $r\ge 3$ (and $s\ge 3$) is the only dimensional assumption,
as in~\cite{LeeSahin2014,Kumar2020}; Example~\ref{ex:subm1}, whose
$\phi_1$-active vertical part is three-dimensional, realizes a pointwise slant
submersion with odd active dimension.
\end{remark}

\section{Pointwise Slant Riemannian Submersions}\label{sec:pointwise_slant}

This section develops the fundamental properties of pointwise slant structures
for Riemannian submersions. Two settings arise, according to whether the total
space carries an almost complex or an almost contact metric structure.
\begin{enumerate}[label=(\roman*),start=1]
  \item Riemannian submersions $F:(M_1,g_1,J_1)\to(M_2,g_2)$ from an almost
        Hermitian manifold \cite{LeeSahin2014,Sahin2013};
  \item Riemannian submersions
        $F:(M_1,g_1,\phi_1,\xi_1,\eta_1)\to(M_2,g_2)$ from an almost contact
        metric manifold \cite{LeeSahin2014,Kumar2020,Prasad2021}.
\end{enumerate}
In each case, the structure tensor ($J_1$ or $\phi_1$) is applied to a vector
from a chosen distribution the vertical $\ker F_*$ or the horizontal
$(\ker F_*)^{\perp}$ and the result is split orthogonally into components
inside and outside that distribution. Throughout, $U$ denotes a vector field in
$\ker F_*$ (vertical) and $X$ a vector field in $(\ker F_*)^{\perp}$
(horizontal); $P$ denotes the component of the image {inside} the chosen
distribution and $Q$ the component in its orthogonal complement. Both settings
yield the same key norm identity, collected in
Proposition~\ref{prop:key_slant_identities}.

\subsection{ Submersion from an almost Hermitian manifold}
Let $F:(M_1,g_1,J_1)\to(M_2,g_2)$ be a Riemannian submersion from $(M_1,g_1,J_1)$ onto $(M_2,g_2)$.  Since $TM_1=\ker F_*\oplus
(\ker F_*)^\perp$, applying $J_1$ to a vertical vector $U$
gives a vector that splits into vertical and horizontal parts,
and applying $J_1$ to a horizontal vector $X$ likewise splits:

\medskip
For any vector field $U\in\Gamma(\ker F_*)$
\begin{equation}\label{eq:J_decomposition_vertical}
  J_1 U
  = P^{\mathcal{V}}U + Q^{\mathcal{V}}U,\qquad
  P^{\mathcal{V}}U\in\Gamma(\ker F_*),\quad
  Q^{\mathcal{V}}U\in\Gamma(\ker F_*)^{\perp},
\end{equation}
where the linear operators $P^{\mathcal{V}} : \ker F_*\;\to\; \ker F_*,$ and $Q^{\mathcal{V}} : \ker F_*\;\to\; (\ker F_*)^{\perp},$ i.e., $P^{\mathcal{V}}U$ is the part of $J_1 U$ that stays
vertical; $Q^{\mathcal{V}}U$ is the part that becomes
horizontal.

\medskip
Similarly, for $X\in\Gamma((\ker F_*)^{\perp})$
\begin{equation}\label{eq:J_decomposition_horizontal}
  J_1 X
  = P^{\mathcal{H}}X + Q^{\mathcal{H}}X,\qquad
  P^{\mathcal{H}}X\in\Gamma(\ker F_*)^{\perp},\quad
  Q^{\mathcal{H}}X\in\Gamma(\ker F_*),
\end{equation}
where the linear operators
 $ P^{\mathcal{H}} : (\ker F_*)^{\perp}
    \;\to\; (\ker F_*)^{\perp},$ and 
  $Q^{\mathcal{H}} : (\ker F_*)^{\perp}
    \;\to\; \ker F_*,$ i.e., 
$P^{\mathcal{H}}X$ is the part of $J_1 X$ that stays
horizontal; $Q^{\mathcal{H}}X$ is the part that becomes
vertical.

\begin{remark}
These are two genuinely different linear operators with
different domains.  $P^{\mathcal{V}}:\ker F_*\to\ker F_*$
measures how much $J_1$ preserves the vertical distribution,
while $P^{\mathcal{H}}:(\ker F_*)^\perp\to(\ker F_*)^\perp$
measures how much $J_1$ preserves the horizontal distribution.
\end{remark}

\begin{definition}[\cite{LeeSahin2014}]
\label{def:ps_subm_hermitian}
A Riemannian submersion $F:(M_1,g_1,J_1)\to(M_2,g_2)$ from an
almost Hermitian manifold $(M_1,g_1,J_1)$ is called a pointwise slant Riemannian submersion if there exists a function
$\theta:M_1\to[0,\pi/2]$ such that, at each fixed $p\in M_1$, the
angle $\theta(p)$ between $J_1U$ and $\mathcal{V}_p$ is independent
of the choice of nonzero $U\in\mathcal{V}_p$ --- but $\theta(p)$ may
vary from point to point. Equivalently,
\begin{equation*}
  \cos\theta(p) = \frac{\|P^{\mathcal{V}}U\|}{\|U\|}
\end{equation*}
for every nonzero $U\in\mathcal{V}_p$, i.e.\ ${(P^{\mathcal V})}^2=-\cos^2\theta(p)\,I$
with $\theta(p)$ referred to as a slant function.

\end{definition}

Since $J_1^2=-I$, for the linear operator $P^{\mathcal V}$ one has
\begin{equation}\label{eq:ids_subm_hermitian}
  (P^{\mathcal{V}})^2 = -\cos^2\!\theta\cdot I,\qquad
  g_1(P^{\mathcal{V}}U,\,P^{\mathcal{V}}U')
    = \cos^2\!\theta\cdot g_1(U,U'),\qquad
  g_1(Q^{\mathcal{V}}U,\,Q^{\mathcal{V}}U')
    = \sin^2\!\theta\cdot g_1(U,U'),
\end{equation}
for all $U,U'\in\Gamma(\ker F_*)$ and the analogous identities hold for the
linear operator $P^{\mathcal H}$. By Parseval's identity applied to an
orthonormal basis $\{e_1,\ldots,e_r\}$ of $\ker F_*$ these give, for the
{vertical} case,
\begin{equation}\label{eq:P_norm_slant_subm_hermitian}
  \|P^{\mathcal{V}}\|^2
  \;:=\;\sum_{i,j=1}^{r}
    \bigl(g_1(e_i,\,P^{\mathcal{V}}e_j)\bigr)^2
  \;=\; r\cos^2\theta,
  \qquad r=\dim\ker F_*,
\end{equation}
and analogously $\|P^{\mathcal{H}}\|^2=s\cos^2\theta$ with $s=\dim(\ker F_*)^\perp$.

\subsection{Submersion from an almost contact metric manifold}

Let $F:(M_1,g_1,\phi_1,\xi_1,\eta_1)\to(M_2,g_2)$ be a Riemannian submersion from an almost contact metric manifold $(M_1,g_1,\phi_1,\xi_1,\eta_1)$  onto a Riemannian manifold $(M_2,g_2)$. Applying $\phi_1$ to a vertical vector
$U$ and to a horizontal vector $X$ produces two separate
decompositions.\\

For any vector field $U\in\Gamma(\ker F_*)$
\begin{equation}\label{eq:phi1_decomp_V}
  \phi_1 U
  = P^{\mathcal{V}}U + Q^{\mathcal{V}}U,\qquad
  P^{\mathcal{V}}U\in\Gamma(\ker F_*),\quad
  Q^{\mathcal{V}}U\in\Gamma((\ker F_*)^{\perp}).
\end{equation}
$P^{\mathcal{V}}U$ is the part of $\phi_1 U$ that stays
vertical; $Q^{\mathcal{V}}U$ is the part that becomes
horizontal. Similarly, for any vector field $X\in\Gamma(\ker F_*)^{\perp}$
\begin{equation}\label{eq:phi1_decomp_H}
  \phi_1 X
  = P^{\mathcal{H}}X + Q^{\mathcal{H}}X,\qquad
  P^{\mathcal{H}}X\in\Gamma(\ker F_*)^{\perp},\quad
  Q^{\mathcal{H}}X\in\Gamma(\ker F_*).
\end{equation}
$P^{\mathcal{H}}X$ is the part of $\phi_1 X$ that stays
horizontal; $Q^{\mathcal{H}}X$ is the part that becomes
vertical. The above operators $P^{\mathcal{V}},~ Q^{\mathcal{V}},~P^{\mathcal{H}}, Q^{\mathcal{H}}$ are defined exactly as in subsection 3.1, with $\phi_1$ in place of
$J_1$.
\medskip

\begin{definition}\cite{Kumar2020,Prasad2021,LeeSahin2014}
Let $F:(M_1,g_1,\phi_1,\xi_1,\eta_1)\to(M_2,g_2)$ be a Riemannian
submersion from an almost contact metric manifold. Then $F$ is called
pointwise slant if
there exists $\theta:M_1\to[0,\pi/2]$ such that, at each $p\in M_1$,
the angle $\theta(p)$ between $\phi_1 U$ and $\mathcal{V}_p$ is
independent of the choice of nonzero
$U\in\ker F_{*p}\setminus\{\xi_1\}$. Equivalently,
$\cos\theta(p)=\dfrac{\|P^{\mathcal{V}}U\|}{\|U\|}$, i.e.\
$(P^{\mathcal{V}})^2=-\cos^2\theta\,I$ on $\ker F_*\setminus\{\xi_1\}$.

\end{definition}

The algebraic identities follow from $\phi_1^2=-I$ on $\ker\eta_1$
by the polarization argument of \cite[Lemma~2.2]{Park2015}. Write
$r=\dim\ker F_*$ and $s=\dim(\ker F_*)^{\perp}$ for the full vertical
and horizontal dimensions. Since $\phi_1$ annihilates $\xi_1$, the slant
condition is active on the orthogonal complement of $\xi_1$ in the
relevant distribution, so the effective dimension drops by one in
whichever distribution contains $\xi_1$. Parseval's identity then gives
\begin{equation}\label{eq:norm_identity_contact}
  \|P^{\mathcal{V}}\|^2 =
  \begin{cases}
    (r-1)\cos^2\theta, & \xi_1\in\Gamma(\ker F_*),\\[2pt]
    r\cos^2\theta, & \xi_1\in\Gamma((\ker F_*)^{\perp}),
  \end{cases}
  \qquad
  \|P^{\mathcal{H}}\|^2 =
  \begin{cases}
    (s-1)\cos^2\theta, & \xi_1\in\Gamma((\ker F_*)^{\perp}),\\[2pt]
    s\cos^2\theta, & \xi_1\in\Gamma(\ker F_*).
  \end{cases}
\end{equation}

\begin{proposition}\label{prop:key_slant_identities}
Let $\mathcal{D}_\theta$ denote the distribution on which the slant
condition is imposed --- the vertical or horizontal distribution in the
almost Hermitian case, and its intersection with $\ker\eta_1$ in the
almost contact case --- and set $d=\dim\mathcal{D}_\theta$. Then the
tangential projection operator $P$ satisfies
\begin{equation}\label{eq:norm_P}
  \|P\|^2 := \sum_{i,j}\bigl(g_1(e_i,Pe_j)\bigr)^2 = d\cos^2\theta
\end{equation}
for any orthonormal basis $\{e_i\}$ of $\mathcal{D}_\theta$. This is
\eqref{eq:P_norm_slant_subm_hermitian} in the almost Hermitian case
($d=r$ or $s$) and \eqref{eq:norm_identity_contact} in the almost
contact case, where $d=r-1$ or $r$ (resp.\ $s-1$ or $s$) according to the
position of $\xi_1$.
\end{proposition}

\begin{remark}\label{rem:special_cases}
Proposition~\ref{prop:key_slant_identities} encompasses the
following classical special cases
\cite{Chen1990,Sahin2010a,LeeSahin2014}
\begin{enumerate}[label={(\alph*)}]
  \item \textbf{Invariant case ($\theta=0$):}\;
        $\cos\theta=1$, so $P=J$ (or $P=\phi$ on $\ker\eta_1$)
        and $Q=0$; the structure map preserves the chosen
        distribution entirely, $\|P\|^2=r$.
  \item \textbf{Anti-invariant case ($\theta=\pi/2$):}\;
        $\cos\theta=0$, so $P=0$ and $Q=J$ (or $Q=\phi$); the
        structure map sends every vector entirely out of the
        chosen distribution, $\|P\|^2=0$.
  \item \textbf{Proper slant ($0<\theta<\pi/2$):}\; both $P$
        and $Q$ are nontrivial and
        $\|P\|^2=r\cos^2\theta\in(0,r)$ interpolates
        continuously between the two extremes.
\end{enumerate}
The Casorati inequalities in Sections~4 and~5 therefore recover
all previously known results for the invariant and
anti-invariant cases simultaneously through the single
parameter $\cos^2\theta$.
\end{remark}

With the fundamental properties of pointwise slant Riemannian submersions now
established, we proceed to derive our main Casorati inequalities. The central
role in every proof is played by the norm
identities~\eqref{eq:P_norm_slant_subm_hermitian} and
\eqref{eq:norm_identity_contact}, which express $\|P\|^2$ purely in terms of
$\cos^2\theta$ and thereby connect the extrinsic Casorati curvature to the slant
angle.

\section{Casorati Inequalities for the Vertical Distribution}\label{sec:vertical}

Riemannian submersions possess two natural distributions, the vertical
$\calV=\ker F_*$ and the horizontal $\calH=(\ker F_*)^{\perp}$, governed by the
O'Neill tensors $T$ and $A$, respectively. In this section we treat the vertical
distribution, whose extrinsic geometry is encoded by $T$. Throughout we assume
$r=\dim\calV\geq 3$.

Let 
$\{U_1,\ldots,U_r\}$ be an orthonormal basis of $\calV$ 
at $p \in M_1$ and let $\{U_{r+1},\ldots,U_{m_1}\}$ be 
an orthonormal basis of $\calH$ at $p$. The components of 
O'Neill's tensor $T$ are
\begin{equation}\label{eq:T_components}
T_{ij}^{\alpha} = g_1(T_{U_i}U_j,\, U_\alpha), \quad 
i,j = 1,\ldots,r,\; \alpha = r+1,\ldots,m_1,
\end{equation}
and the squared norm of $T$ restricted to $\calV$ is
\begin{equation}\label{eq:T_norm}
\|T\|^2 = \sum_{\alpha=r+1}^{m_1}\sum_{i,j=1}^r 
           (T_{ij}^{\alpha})^2.
\end{equation}
The {Casorati curvature of the vertical distribution} 
and its normalized variants are
\begin{equation}\label{eq:casorati_vertical}
C^{\calV} = \dfrac{1}{r}\|T\|^2, \qquad
C(L) = \dfrac{1}{k}\sum_{\alpha=r+1}^{m_1}
        \sum_{i,j=1}^k (T_{ij}^{\alpha})^2
\end{equation}
for a $k$-dimensional subspace $L \subset \calV$ 
$(k \ge 2)$ with orthonormal basis $\{U_1,\ldots,U_k\}$,
\begin{align}
\delta_C^{\calV}(r-1) 
&= \dfrac{1}{2}C^{\calV} 
 + \dfrac{r+1}{2r}
   \inf\{C(L) : L \text{ is a hyperplane of } \calV\},
\label{eq:delta_casorati_vert}\\[4pt]
\hat{\delta}_C^{\calV}(r-1) 
&= 2C^{\calV} 
 - \dfrac{2r-1}{2r}
   \sup\{C(L) : L \text{ is a hyperplane of } \calV\}.
\label{eq:delta_hat_casorati_vert}
\end{align}
The fiber scalar curvature $\tau^{\calV}$, computed from 
the curvature tensor $\hat{R}$ of the fibers, and the 
ambient scalar curvature $\tau_{M_1}^{\calV}$, computed 
from $R^{M_1}$, have normalized forms
\begin{equation}\label{eq:normalized_scalar_vert}
\rho^{\calV} = \dfrac{2\tau^{\calV}}{r(r-1)}, \qquad
\rho_{M_1}^{\calV} 
= \dfrac{2\tau_{M_1}^{\calV}}{r(r-1)}.
\end{equation}
The equality condition for the case of vertical distribution is 
characterized by invariant quasi-umbilicity of the fibers: 
there exist orthonormal bases such that 
$T_{ij}^\alpha = 0$ for $i \neq j$ and 
$T_{11}^\alpha = \cdots = T_{r-1,r-1}^\alpha 
= \dfrac{1}{2}T_{rr}^\alpha$ for all $\alpha$, in complete 
analogy with Definition~\ref{def:quasi-umbilical}.\\

The equality cases of our vertical inequalities are characterized by the
following notion, which generalizes classical quasi-umbilical submanifolds to
the fibers of a submersion.

\begin{definition}\label{def:quasi-umbilical}
The fibers of a Riemannian submersion $F:(M_1,g_1)\to(M_2,g_2)$ are called
{invariantly quasi-umbilical} at $p\in M_1$ if there exist orthonormal
bases $\{U_1,\ldots,U_r\}$ of $\calV=\ker F_{*p}$ and
$\{U_{r+1},\ldots,U_{m_1}\}$ of $\calH=(\ker F_{*p})^{\perp}$ such that the
components $T_{ij}^{\alpha}=g_1(T_{U_i}U_j,U_\alpha)$ of O'Neill's tensor
satisfy
\begin{enumerate}[label=(\roman*)]
\item $T_{ij}^\alpha = 0$ for all $i \neq j$ and all
      $\alpha = r+1, \ldots, m_1$;
\item $T_{11}^\alpha = T_{22}^\alpha = \cdots =
      T_{r-1,r-1}^\alpha$ for all $\alpha = r+1, \ldots, m_1$.
\end{enumerate}
\end{definition}

\begin{remark}\label{rem:quasi-umbilical}
Condition (i) ensures that $T$ is diagonal in suitable bases, while condition
(ii) requires the first $r-1$ principal components to be equal. The equality
case in Theorem~\ref{thm:complex_subm_vert} below corresponds to the sharper
requirement
\[
T_{11}^\alpha = T_{22}^\alpha = \cdots =
T_{r-1,r-1}^\alpha = \dfrac{1}{2}\,T_{rr}^\alpha, \quad
T_{ij}^\alpha = 0 \text{ for } i \neq j,
\]
for all $\alpha = r+1, \ldots, m_1$.
\end{remark}

\subsection{Generalized complex space forms}

Our first result for submersions gives an explicit bound when the 
total space is a generalized complex space form.
\begin{theorem}\label{thm:complex_subm_vert}
Let $F:(M_1^{m_1}(c_1,c_2),g_1,J_1)\to (M_2^{m_2},g_2)$ be a pointwise slant Riemannian submersion from a generalized complex space form onto a Riemannian manifold with vertical space of dimension $r\ge 3$ and slant function $\theta$. Then
\begin{equation}\label{eq:complex_subm_vert_delta}
\rho^{\calV} \le \delta_C^{\calV}(r-1) + c_1 + \dfrac{3c_2}{r(r-1)}\|P^{\mathcal V}\|^2 = \delta_C^{\calV}(r-1) + c_1 + \dfrac{3c_2}{r-1}\cos^2\theta,
\end{equation}
and
\begin{equation}\label{eq:complex_subm_vert_delta_hat}
\rho^{\calV} \le \hat{\delta}_C^{\calV}(r-1) + c_1 + \dfrac{3c_2}{r(r-1)}\|P^{\mathcal V}\|^2 = \hat{\delta}_C^{\calV}(r-1) + c_1 + \dfrac{3c_2}{r-1}\cos^2\theta.
\end{equation}
Equality holds at $p \in M_1$ if and only if for suitable orthonormal 
bases, the components of the $T$-tensor satisfy
\[
T_{11}^{\alpha} = T_{22}^{\alpha} = \cdots = T_{r-1,r-1}^{\alpha} 
= \dfrac{1}{2}T_{rr}^{\alpha}, \quad T_{ij}^{\alpha} = 0 \text{ for } i\neq j,
\]
for all $\alpha = r+1,\ldots,m_1$. Geometrically, this means that the 
fibers of $F$ are invariantly quasi-umbilical in the sense of 
Definition~\ref{def:quasi-umbilical}. 
\end{theorem}

\begin{proof}
Using the curvature formula \eqref{eq:gen_complex_curvature} for $M_1(c_1,c_2)$ and an orthonormal basis $\{U_1,\ldots,U_r\}$ of $\ker F_*$, we compute
\begin{align*}
2\tau_{M_1}^{\calV} &= \sum_{i,j=1}^r g_1(R^{M_1}(U_i,U_j)U_j,U_i)\\
&= \sum_{i,j=1}^r c_1\{g_1(U_j,U_j)g_1(U_i,U_i) - g_1(U_i,U_j)g_1(U_j,U_i)\}\\
&\quad + \sum_{i,j=1}^r c_2\{g_1(U_i, J_1U_j)g_1(J_1U_j, U_i) - g_1(U_j, J_1U_j)g_1(J_1U_i, U_i)\\
&\quad + 2g_1(U_i, J_1U_j)g_1(J_1U_j, U_i)\}.
\end{align*}

The first sum gives $r(r-1)c_1$ by orthonormality. For the second sum, we use the decomposition \eqref{eq:J_decomposition_vertical} to write $J_1U_j = P^{\mathcal V}U_j + Q^{\mathcal V}U_j$ where $P^{\mathcal V}U_j \in \ker F_*$ and $Q^{\mathcal V}U_j \in (\ker F_*)^\perp$. Since $J_1$ is an almost complex structure, $g_1(U_j, J_1U_j) = 0$, and we obtain
\[\sum_{i,j=1}^r c_2\{g_1(U_i, P^{\mathcal V}U_j)g_1(P^{\mathcal V}U_j, U_i) + 2g_1(U_i, P^{\mathcal V}U_j)g_1(P^{\mathcal V}U_j, U_i)\} = 3c_2\|P^{\mathcal V}\|^2,\]
where $\|P^{\mathcal V}\|^2 = \sum_{i,j=1}^r (g_1(U_i, P^{\mathcal V}U_j))^2$ as in equation \eqref{eq:P_norm_slant_subm_hermitian}.

Therefore, $2\tau_{M_1}^{\calV} = r(r-1)c_1 + 3c_2\|P^{\mathcal V}\|^2$, which gives
\[\rho_{M_1}^{\calV} = c_1 + \dfrac{3c_2}{r(r-1)}\|P^{\mathcal V}\|^2.\]

Applying the algebraic optimization of Lemma~\ref{lem:tripathi}
to the vertical Gauss equation~\eqref{eq:gauss_vertical}, we obtain
\[
\rho^{\calV} \leq \delta_C^{\calV}(r-1) + \rho_{M_1}^{\calV}.
\]
Substituting $\rho_{M_1}^{\calV}$ and using 
$\|P^{\calV}\|^2 = r\cos^2\theta$ from 
equation~\eqref{eq:P_norm_slant_subm_hermitian}, 
we obtain the desired inequalities 
\eqref{eq:complex_subm_vert_delta} 
and~\eqref{eq:complex_subm_vert_delta_hat}.

\medskip
Equality holds if and only if the components of the $T$-tensor satisfy the condition stated in the theorem, which follows from Lemma~\ref{lem:tripathi} with parameters $\lambda_1 = r$ and $\lambda_2 = \dfrac{r-1}{2}$.
\end{proof}

\begin{corollary}\label{cor:subm_vert_real}
Let $F:(M_1^{m_1},g_1)\to (M_2^{m_2},g_2)$ be a pointwise slant Riemannian submersion with vertical space of dimension $r\ge 3$ and slant function $\theta$.
\begin{enumerate}
    \item If $M_1(c)$ is a real space form, then
    \begin{equation}\label{eq:subm_vert_real}
    \rho^{\calV} \le \delta_C^{\calV}(r-1) + c, \quad \rho^{\calV} \le \hat{\delta}_C^{\calV}(r-1) + c.
    \end{equation}
    
    \item If $M_1(c)$ is a complex space form, then
    \begin{align}
    \rho^{\calV} &\le \delta_C^{\calV}(r-1) + \dfrac{c}{4}\left(1 + \dfrac{3}{r-1}\cos^2\theta\right), \label{eq:subm_vert_complex}\\
    \rho^{\calV} &\le \hat{\delta}_C^{\calV}(r-1) + \dfrac{c}{4}\left(1 + \dfrac{3}{r-1}\cos^2\theta\right).
    \end{align}
    
    \item If $M_1(c)$ is a real K\"{a}hler space form, then
    \begin{align}
    \rho^{\calV} &\le \delta_C^{\calV}(r-1) + \dfrac{c+3\alpha}{4} + \dfrac{3(c-\alpha)}{4(r-1)}\cos^2\theta, \label{eq:subm_vert_kahler}\\
    \rho^{\calV} &\le \hat{\delta}_C^{\calV}(r-1) + \dfrac{c+3\alpha}{4} + \dfrac{3(c-\alpha)}{4(r-1)}\cos^2\theta.
    \end{align}
\end{enumerate}
\end{corollary}

The proofs are immediate by substitution.\\ 

\subsection{Generalized Sasakian space forms}

When the total space is a generalized Sasakian space form, we must again account for the Reeb vector field, leading to two cases depending on whether $\xi_1 \in \ker F_*$ or $\xi_1 \in (\ker F_*)^\perp$.

\begin{theorem}\label{thm:sasakian_subm_vert}
Let $F:(M_1^{m_1}(c_1,c_2,c_3),\phi_1,\xi_1,\eta_1,g_1)\to (M_2^{m_2},g_2)$ be a pointwise slant Riemannian submersion from a generalized Sasakian space form onto a Riemannian manifold with vertical space of dimension $r\ge 3$ and slant function $\theta$. Then
\begin{equation}\label{eq:sasakian_subm_vert}
\rho^{\calV} \le \begin{cases}
\delta_C^{\calV}(r-1) + c_1 + \dfrac{3c_2}{r}\cos^2\theta - \dfrac{2}{r}c_3, & \text{if } \xi_1 \in \Gamma(\ker F_*),\\[6pt]
\delta_C^{\calV}(r-1) + c_1 + \dfrac{3c_2}{r-1}\cos^2\theta, & \text{if } \xi_1 \in \Gamma(\ker F_*)^\perp,
\end{cases}
\end{equation}
and similarly for $\hat{\delta}_C^{\calV}$. Equality holds if and only if the fibers are invariantly quasi-umbilical 
 which is the same as in 
Theorem~\ref{thm:complex_subm_vert}.
\end{theorem}
\begin{proof}
The computation follows Theorem~\ref{thm:complex_subm_vert} for the
$c_1,c_2$ terms, with the additional $c_3$ (Reeb) term
in~\eqref{eq:gen_sasakian_curvature} contributing according to whether
$\xi_1$ is vertical or horizontal. By~\eqref{eq:norm_identity_contact},
$\|P^{\calV}\|^2=(r-1)\cos^2\theta$ when $\xi_1\in\Gamma(\ker F_*)$ and
$\|P^{\calV}\|^2=r\cos^2\theta$ when $\xi_1\in\Gamma((\ker F_*)^{\perp})$.
Hence
\[
\rho_{M_1}^{\calV}
= c_1+\frac{3c_2}{r(r-1)}\,(r-1)\cos^2\theta-\frac{2c_3}{r}
= c_1+\frac{3c_2}{r}\cos^2\theta-\frac{2c_3}{r},
\qquad \xi_1\in\Gamma(\ker F_*),
\]
\[
\rho_{M_1}^{\calV}
= c_1+\frac{3c_2}{r(r-1)}\,r\cos^2\theta
= c_1+\frac{3c_2}{r-1}\cos^2\theta,
\qquad \xi_1\in\Gamma((\ker F_*)^{\perp}).
\]
Substituting into $\rho^{\calV}\leq\delta_C^{\calV}(r-1)
+\rho_{M_1}^{\calV}$ obtains the result.
\end{proof}

As in the complex case, specializing
Theorem~\ref{thm:sasakian_subm_vert} to the standard contact space forms
gives the following explicit inequalities.

\begin{corollary}\label{cor:subm_vert_sasakian}
Let $F:(M_1^{m_1},g_1)\to (M_2^{m_2},g_2)$ be a pointwise 
slant Riemannian submersion with vertical space of dimension 
$r\ge 3$ and slant function $\theta$.

\begin{enumerate}
    \item If $M_1(c)$ is a Sasakian space form 
    ($c_1=\dfrac{c+3}{4}$, $c_2=c_3=\dfrac{c-1}{4}$), then
    \begin{align*}
    \rho^{\mathcal{V}} \le 
    \begin{cases}
    \delta_C^{\mathcal{V}}(r-1) + \dfrac{c+3}{4} 
    + \dfrac{3(c-1)}{4r}\cos^2\theta 
    - \dfrac{c-1}{2r}, 
    & \xi_1 \in \Gamma(\ker F_*),\\[8pt]
    \delta_C^{\mathcal{V}}(r-1) + \dfrac{c+3}{4} 
    + \dfrac{3(c-1)}{4(r-1)}\cos^2\theta,
    & \xi_1 \in \Gamma((\ker F_*)^{\perp}),
    \end{cases}
    \end{align*}
    and similarly for $\hat{\delta}_C^{\mathcal{V}}(r-1)$.

    \item If $M_1(c)$ is a Kenmotsu space form 
    ($c_1=\dfrac{c-3}{4}$, $c_2=c_3=\dfrac{c+1}{4}$), then
    \begin{align*}
    \rho^{\mathcal{V}} \le 
    \begin{cases}
    \delta_C^{\mathcal{V}}(r-1) + \dfrac{c-3}{4} 
    + \dfrac{3(c+1)}{4r}\cos^2\theta 
    - \dfrac{c+1}{2r}, 
    & \xi_1 \in \Gamma(\ker F_*),\\[8pt]
    \delta_C^{\mathcal{V}}(r-1) + \dfrac{c-3}{4} 
    + \dfrac{3(c+1)}{4(r-1)}\cos^2\theta,
    & \xi_1 \in \Gamma((\ker F_*)^{\perp}),
    \end{cases}
    \end{align*}
    and similarly for $\hat{\delta}_C^{\mathcal{V}}(r-1)$.

    \item If $M_1(c)$ is a cosymplectic space form 
    ($c_1=c_2=c_3=\dfrac{c}{4}$), then
    \begin{align*}
    \rho^{\mathcal{V}} \le 
    \begin{cases}
    \delta_C^{\mathcal{V}}(r-1) + \dfrac{c}{4} 
    + \dfrac{3c}{4r}\cos^2\theta 
    - \dfrac{c}{2r}, 
    & \xi_1 \in \Gamma(\ker F_*),\\[8pt]
    \delta_C^{\mathcal{V}}(r-1) + \dfrac{c}{4} 
    + \dfrac{3c}{4(r-1)}\cos^2\theta,
    & \xi_1 \in \Gamma((\ker F_*)^{\perp}),
    \end{cases}
    \end{align*}
    and similarly for $\hat{\delta}_C^{\mathcal{V}}(r-1)$.

    \item If $M_1(c)$ is an almost $C(\alpha)$ space form 
    ($c_1=\dfrac{c+3\alpha^2}{4}$, 
    $c_2=c_3=\dfrac{c-\alpha^2}{4}$), then
    \begin{align*}
    \rho^{\mathcal{V}} \le 
    \begin{cases}
    \delta_C^{\mathcal{V}}(r-1) + \dfrac{c+3\alpha^2}{4} 
    + \dfrac{3(c-\alpha^2)}{4r}\cos^2\theta 
    - \dfrac{c-\alpha^2}{2r}, 
    & \xi_1 \in \Gamma(\ker F_*),\\[8pt]
    \delta_C^{\mathcal{V}}(r-1) + \dfrac{c+3\alpha^2}{4} 
    + \dfrac{3(c-\alpha^2)}{4(r-1)}\cos^2\theta,
    & \xi_1 \in \Gamma((\ker F_*)^{\perp}),
    \end{cases}
    \end{align*}
    and similarly for $\hat{\delta}_C^{\mathcal{V}}(r-1)$.
\end{enumerate}
In all cases, the equality conditions are the same as in 
Theorem~\ref{thm:complex_subm_vert}
\end{corollary}

\section{Casorati Inequalities for the Horizontal Distribution}\label{sec:horizontal}

We now turn to the horizontal distribution $\calH=(\ker F_*)^{\perp}$, where the
geometric picture is quite different. For horizontal distributions O'Neill's
tensor $A$ plays the central role, and the equality condition in the Casorati
inequality corresponds to integrability of the horizontal distribution rather
than to quasi-umbilicity. Recall from~\eqref{eq:tensor_A} that $A$ measures the
failure of $\calH$ to be integrable; when $A=0$, the horizontal distribution is
integrable by Frobenius' theorem. Throughout we assume $\dim\calH\geq 3$.

 Let 
$\{X_1,\ldots,X_s\}$ be an orthonormal basis of $\calH$ 
at $p$, where $s = \dim\calH \geq 3$. The components of 
O'Neill's tensor $A$ are
\begin{equation}\label{eq:A_components}
A_{ij}^{k} = g_1(A_{X_i}X_j,\, U_k), \quad 
i,j = 1,\ldots,s,\; k = 1,\ldots,r,
\end{equation}
with squared norm
\begin{equation}\label{eq:A_norm}
\|A\|^2 = \sum_{k=1}^{r}\sum_{i,j=1}^s 
           (A_{ij}^{k})^2.
\end{equation}
The Casorati curvature $C^{\calH}$ and normalized variants 
$\delta_C^{\calH}(s-1)$, $\hat{\delta}_C^{\calH}(s-1)$ 
are defined in the same manner with 
\eqref{eq:casorati_vertical}--\eqref{eq:delta_hat_casorati_vert}, 
replacing $T_{ij}^\alpha$ with $A_{ij}^k$ and $r$ with 
$s$. The scalar curvature of the horizontal distribution 
and its normalized form are
\begin{equation}\label{eq:normalized_scalar_horiz}
\tau^{\calH}(p) 
= \dfrac{1}{2}\sum_{i,j=1}^s 
  g_1(R^{M_1}(X_i,X_j)X_j,X_i), \qquad
\rho^{\calH} = \dfrac{2\tau^{\calH}}{s(s-1)},
\end{equation}
where $\rho_{\calH}^{\calH}$ denotes the normalized scalar 
curvature computed from the curvature tensor $R^{\calH}$ 
of the horizontal distribution itself. The equality 
condition for the case of horizontal distribution is $A \equiv 0$, 
which by Proposition~\ref{prop:oneill_tensors}(iv) is 
equivalent to complete integrability of $\calH$.

\subsection{Generalized complex space forms}

\begin{theorem}\label{thm:complex_subm_horiz}
Let $F:(M_1^{m_1}(c_1,c_2),g_1,J_1)\to (M_2^{m_2},g_2)$ be a pointwise slant Riemannian submersion from a generalized complex space form onto a Riemannian manifold with horizontal space of dimension $s\ge 3$ and slant function $\theta$. Then
\begin{equation}\label{eq:complex_horiz}
\rho_{\calH}^{\calH} \le \delta_C^{\calH}(s-1) + c_1 + \dfrac{3c_2}{s-1}\cos^2\theta,
\end{equation}
and similarly for $\hat{\delta}_C^{\calH}$.Equality holds if and only if the tensor $A$ vanishes identically, i.e., 
$A_XY = 0$ for all $X,Y \in \Gamma((\ker F_*)^\perp)$, which is equivalent 
to the horizontal distribution being completely integrable.
\end{theorem}

\begin{proof}
Using the curvature formula~\eqref{eq:gen_complex_curvature} 
for $M_1(c_1,c_2)$ and an orthonormal basis 
$\{X_1,\ldots,X_s\}$ of $(\ker F_*)^\perp$, the same 
computation as in Theorem~\ref{thm:complex_subm_vert} 
with decomposition~\eqref{eq:J_decomposition_horizontal} 
and $g_1(X_j, J_1X_j) = 0$ gives
\begin{equation}\label{eq:rho_H_complex_horiz}
\rho^{\calH} = c_1 + \dfrac{3c_2}{s(s-1)}\|P^{\calH}\|^2
= c_1 + \dfrac{3c_2}{s-1}\cos^2\theta,
\end{equation}
where the second equality uses 
$\|P^{\calH}\|^2 = s\cos^2\theta$ from 
equation~\eqref{eq:P_norm_slant_subm_hermitian}.\\

Setting $X_1 = X_4 = X_i$ and $X_2 = X_3 = X_j$ in 
the horizontal Gauss equation~\eqref{eq:gauss_horizontal} 
and summing over $i,j = 1,\ldots,s$, we obtain
\[
2\tau^{\calH} = 2\tau^{\calH}_{\calH} 
+ 3s\,C^{\calH} - \|\operatorname{tr} A^{\calH}\|^2,
\]
where $\operatorname{tr} A^{\calH} = \sum_i A_{X_i}X_i$. 
Since $A$ is skew-symmetric by 
Proposition~\ref{prop:oneill_tensors}(ii), we have 
$A_{X_i}X_i = 0$ for all $i$, so 
$\operatorname{tr} A^{\calH} = 0$ and the equation 
simplifies to
\[
2\tau^{\calH}_{\calH} = 2\tau^{\calH} - 3s\,C^{\calH}.
\]
Constructing the polynomial 
$\mathcal{P}^{\calH} = \dfrac{1}{2}s(s-1)C^{\calH} 
+ \dfrac{1}{2}(s^2-1)C(L^{\calH}) 
+ 2\tau^{\calH} - 2\tau^{\calH}_{\calH}$
and substituting $2\tau^{\calH}_{\calH}=2\tau^{\calH}-3s\,C^{\calH}$, a
direct computation gives
\[
\mathcal{P}^{\calH}
= \dfrac{1}{2}s(s+5)\,C^{\calH}
+ \dfrac{1}{2}(s^2-1)\,C(L^{\calH}).
\]
Since $A$ is skew-symmetric, $A_{ij}^k=0$ whenever $i=j$, so both
$C^{\calH}$ and $C(L^{\calH})$ involve only off-diagonal components and
are non-negative. Hence $\mathcal{P}^{\calH}\ge 0$, with equality if and
only if every $A_{ij}^k=0$. Rearranging and
substituting~\eqref{eq:rho_H_complex_horiz} gives~\eqref{eq:complex_horiz}.

Equality holds if and only if $\mathcal{P}^{\calH} = 0$, 
which requires $A_{ij}^k = 0$ for all $i \neq j$, i.e., 
$A \equiv 0$. By 
Proposition~\ref{prop:oneill_tensors}(iv), this is 
equivalent to complete integrability of $(\ker F_*)^\perp$.
\end{proof}

\begin{corollary}\label{cor:subm_horiz_real}
Let $F:(M_1^{m_1},g_1)\to (M_2^{m_2},g_2)$ be a pointwise slant Riemannian submersion with horizontal space of dimension $s\ge 3$ and slant function $\theta$.
\begin{enumerate}
    \item If $M_1(c)$ is a real space form, then
    \[\rho_{\calH}^{\calH} \le \delta_C^{\calH}(s-1) + c, \quad \rho_{\calH}^{\calH} \le \hat{\delta}_C^{\calH}(s-1) + c.\]
    
    \item If $M_1(c)$ is a complex space form, then
    \[\rho_{\calH}^{\calH} \le \delta_C^{\calH}(s-1) + \dfrac{c}{4}\left(1 + \dfrac{3}{s-1}\cos^2\theta\right),\]
    and similarly for $\hat{\delta}_C^{\calH}$.\\
    
    \item If $M_1(c)$ is a real K\"{a}hler space form, then
    \[\rho_{\calH}^{\calH} \le \delta_C^{\calH}(s-1) + \dfrac{c+3\alpha}{4} + \dfrac{3(c-\alpha)}{4(s-1)}\cos^2\theta,\]
    and similarly for $\hat{\delta}_C^{\calH}$.
\end{enumerate}
\end{corollary}

\subsection{Generalized Sasakian space forms}

Finally, for generalized Sasakian space forms:

\begin{theorem}\label{thm:sasakian_subm_horiz}
Let $F:(M_1^{m_1}(c_1,c_2,c_3),\phi_1,\xi_1,\eta_1,g_1)\to (M_2^{m_2},g_2)$ be a pointwise slant Riemannian submersion from a generalized Sasakian space form onto a Riemannian manifold with horizontal space of dimension $s\ge 3$ and slant function $\theta$. Then
\begin{equation}\label{eq:sasakian_horiz}
\rho_{\calH}^{\calH} \le \begin{cases}
\delta_C^{\calH}(s-1) + c_1 + \dfrac{3c_2}{s}\cos^2\theta - \dfrac{2}{s}c_3, & \text{if } \xi_1 \in \Gamma((\ker F_*)^\perp),\\[6pt]
\delta_C^{\calH}(s-1) + c_1 + \dfrac{3c_2}{s-1}\cos^2\theta, & \text{if } \xi_1 \in \Gamma(\ker F_*),
\end{cases}
\end{equation}
and similarly for $\hat{\delta}_C^{\calH}$. Equality holds if and only if the tensor $A$ vanishes identically, i.e., 
$A_XY = 0$ for all $X,Y \in \Gamma((\ker F_*)^\perp)$, which is equivalent 
to the horizontal distribution being completely integrable.
\end{theorem}
\begin{proof}
The argument parallels Theorem~\ref{thm:complex_subm_horiz} for the
$c_1,c_2$ terms and the optimization of the $A$-tensor, with the
additional $c_3$ (Reeb) term in~\eqref{eq:gen_sasakian_curvature}
contributing exactly as in Theorem~\ref{thm:sasakian_subm_vert},
according to whether $\xi_1$ is horizontal or vertical. By
\eqref{eq:norm_identity_contact},
$\|P^{\calH}\|^2=(s-1)\cos^2\theta$ when $\xi_1\in\Gamma((\ker F_*)^{\perp})$
and $\|P^{\calH}\|^2=s\cos^2\theta$ when $\xi_1\in\Gamma(\ker F_*)$, so
\[
\rho^{\calH}
= c_1+\frac{3c_2}{s(s-1)}\,(s-1)\cos^2\theta-\frac{2c_3}{s}
= c_1+\frac{3c_2}{s}\cos^2\theta-\frac{2c_3}{s},
\qquad \xi_1\in\Gamma((\ker F_*)^{\perp}),
\]
\[
\rho^{\calH}
= c_1+\frac{3c_2}{s(s-1)}\,s\cos^2\theta
= c_1+\frac{3c_2}{s-1}\cos^2\theta,
\qquad \xi_1\in\Gamma(\ker F_*).
\]
Substituting into the horizontal Casorati estimate
$\rho_{\calH}^{\calH}\leq\delta_C^{\calH}(s-1)+\rho^{\calH}$ obtained in
Theorem~\ref{thm:complex_subm_horiz} yields the two cases; the
$A\equiv0$ equality condition is unchanged from that theorem.
\end{proof}

Specializing Theorem~\ref{thm:sasakian_subm_horiz} to the classical
contact space forms yields explicit bounds in each case.
\begin{corollary}\label{cor:subm_horiz_sasakian}
Let $F:(M_1^{m_1},g_1)\to (M_2^{m_2},g_2)$ be a pointwise 
slant Riemannian submersion with horizontal space of dimension 
$s\ge 3$ and slant function $\theta$.

\begin{enumerate}
    \item If $M_1(c)$ is a Sasakian space form 
    ($c_1=\dfrac{c+3}{4}$, $c_2=c_3=\dfrac{c-1}{4}$), then
    \begin{align*}
    \rho_{\mathcal{H}}^{\mathcal{H}} \le 
    \begin{cases}
    \delta_C^{\mathcal{H}}(s-1) + \dfrac{c+3}{4} 
    + \dfrac{3(c-1)}{4s}\cos^2\theta 
    - \dfrac{c-1}{2s}, 
    & \xi_1 \in \Gamma((\ker F_*)^{\perp}),\\[8pt]
    \delta_C^{\mathcal{H}}(s-1) + \dfrac{c+3}{4} 
    + \dfrac{3(c-1)}{4(s-1)}\cos^2\theta,
    & \xi_1 \in \Gamma(\ker F_*),
    \end{cases}
    \end{align*}
    and similarly for $\hat{\delta}_C^{\mathcal{H}}(s-1)$.

    \item If $M_1(c)$ is a Kenmotsu space form 
    ($c_1=\dfrac{c-3}{4}$, $c_2=c_3=\dfrac{c+1}{4}$), then
    \begin{align*}
    \rho_{\mathcal{H}}^{\mathcal{H}} \le 
    \begin{cases}
    \delta_C^{\mathcal{H}}(s-1) + \dfrac{c-3}{4} 
    + \dfrac{3(c+1)}{4s}\cos^2\theta 
    - \dfrac{c+1}{2s}, 
    & \xi_1 \in \Gamma((\ker F_*)^{\perp}),\\[8pt]
    \delta_C^{\mathcal{H}}(s-1) + \dfrac{c-3}{4} 
    + \dfrac{3(c+1)}{4(s-1)}\cos^2\theta,
    & \xi_1 \in \Gamma(\ker F_*),
    \end{cases}
    \end{align*}
    and similarly for $\hat{\delta}_C^{\mathcal{H}}(s-1)$.

    \item If $M_1(c)$ is a cosymplectic space form 
    ($c_1=c_2=c_3=\dfrac{c}{4}$), then
    \begin{align*}
    \rho_{\mathcal{H}}^{\mathcal{H}} \le 
    \begin{cases}
    \delta_C^{\mathcal{H}}(s-1) + \dfrac{c}{4} 
    + \dfrac{3c}{4s}\cos^2\theta 
    - \dfrac{c}{2s}, 
    & \xi_1 \in \Gamma((\ker F_*)^{\perp}),\\[8pt]
    \delta_C^{\mathcal{H}}(s-1) + \dfrac{c}{4} 
    + \dfrac{3c}{4(s-1)}\cos^2\theta,
    & \xi_1 \in \Gamma(\ker F_*),
    \end{cases}
    \end{align*}
    and similarly for $\hat{\delta}_C^{\mathcal{H}}(s-1)$.

    \item If $M_1(c)$ is an almost $C(\alpha)$ space form 
    ($c_1=\dfrac{c+3\alpha^2}{4}$, 
    $c_2=c_3=\dfrac{c-\alpha^2}{4}$), then
    \begin{align*}
    \rho_{\mathcal{H}}^{\mathcal{H}} \le 
    \begin{cases}
    \delta_C^{\mathcal{H}}(s-1) + \dfrac{c+3\alpha^2}{4} 
    + \dfrac{3(c-\alpha^2)}{4s}\cos^2\theta 
    - \dfrac{c-\alpha^2}{2s}, 
    & \xi_1 \in \Gamma((\ker F_*)^{\perp}),\\[8pt]
    \delta_C^{\mathcal{H}}(s-1) + \dfrac{c+3\alpha^2}{4} 
    + \dfrac{3(c-\alpha^2)}{4(s-1)}\cos^2\theta,
    & \xi_1 \in \Gamma(\ker F_*),
    \end{cases}
    \end{align*}
    and similarly for $\hat{\delta}_C^{\mathcal{H}}(s-1)$.
\end{enumerate}
In all cases, the equality conditions are the same as in 
Theorem~\ref{thm:sasakian_subm_horiz}.
\end{corollary}

\section{Examples}\label{sec:examples}

We provide explicit examples illustrating the main theorems and the sharpness of
our inequalities. Example~\ref{ex:subm1} gives a properly slant Riemannian
submersion from a flat cosymplectic total space; Example~\ref{ex:subm2} treats a
warped-product submersion with integrable horizontal distribution; and
Example~\ref{ex:subm3} exhibits a genuinely varying slant function.

\begin{example}\label{ex:subm1}
We construct a pointwise slant Riemannian submersion 
from $\R^7$, equipped with a flat cosymplectic structure.

Let $M_1 = \R^7$ with coordinates 
$(t, x_1, y_1, x_2, y_2, x_3, y_3)$, Reeb vector 
field $\xi = \partial/\partial t$, and almost contact 
structure
\[
\phi\!\left(\dfrac{\partial}{\partial x_i}\right) 
= \dfrac{\partial}{\partial y_i}, \quad
\phi\!\left(\dfrac{\partial}{\partial y_i}\right) 
= -\dfrac{\partial}{\partial x_i}, \quad
\phi(\xi) = 0, \qquad i = 1,2,3.
\]
Fix $\alpha \in (0, \pi/2)$, and define 
$F : \R^7 \to \R^3$ by
\begin{equation}\label{eq:ex3_map}
F(t, x_1, y_1, x_2, y_2, x_3, y_3) 
= \bigl(\sin\alpha\, y_1 - \cos\alpha\, x_3,\;
  -\cos\alpha\, x_1 + \sin\alpha\, y_2,\;
  -\cos\alpha\, x_2 + \sin\alpha\, y_3\bigr).
\end{equation}

The pushforward is
\begin{align*}
F_*\!\left(\dfrac{\partial}{\partial t}\right) 
&= 0, &
F_*\!\left(\dfrac{\partial}{\partial x_1}\right) 
&= -\cos\alpha\, e_2, \\
F_*\!\left(\dfrac{\partial}{\partial y_1}\right) 
&= \sin\alpha\, e_1, &
F_*\!\left(\dfrac{\partial}{\partial x_2}\right) 
&= -\cos\alpha\, e_3, \\
F_*\!\left(\dfrac{\partial}{\partial y_2}\right) 
&= \sin\alpha\, e_2, &
F_*\!\left(\dfrac{\partial}{\partial x_3}\right) 
&= -\cos\alpha\, e_1, \\
F_*\!\left(\dfrac{\partial}{\partial y_3}\right) 
&= \sin\alpha\, e_3, & &
\end{align*}
where $\{e_1, e_2, e_3\}$ is the standard basis of 
$\R^3$. The kernel is
$\ker F_* = \Span\{\xi,\, V_1,\, V_2,\, V_3\}$,
where
\begin{equation}\label{eq:ex3_vertical}
V_1 = \cos\alpha\,
  \dfrac{\partial}{\partial y_1} 
  + \sin\alpha\,
  \dfrac{\partial}{\partial x_3}, \quad
V_2 = \cos\alpha\,
  \dfrac{\partial}{\partial y_2} 
  + \sin\alpha\,
  \dfrac{\partial}{\partial x_1}, \quad
V_3 = \cos\alpha\,
  \dfrac{\partial}{\partial y_3} 
  + \sin\alpha\,
  \dfrac{\partial}{\partial x_2}.
\end{equation}
One verifies $g_1(\xi, V_i) = 0$ and 
$g_1(V_i, V_j) = \delta_{ij}$, so 
$\dim\ker F_* = r = 4$. The horizontal space is
$(\ker F_*)^\perp = \Span\{H_1,\, H_2,\, H_3\}$,
where
\begin{equation}\label{eq:ex3_horizontal}
H_1 = \sin\alpha\,
  \dfrac{\partial}{\partial y_1} 
  - \cos\alpha\,
  \dfrac{\partial}{\partial x_3}, \quad
H_2 = -\cos\alpha\,
  \dfrac{\partial}{\partial x_1} 
  + \sin\alpha\,
  \dfrac{\partial}{\partial y_2}, \quad
H_3 = -\cos\alpha\,
  \dfrac{\partial}{\partial x_2} 
  + \sin\alpha\,
  \dfrac{\partial}{\partial y_3}.
\end{equation}
One checks $g_1(H_i, H_j) = \delta_{ij}$ and 
$F_*(H_i) = e_i$, confirms
$F^*g_{\R^3} = g_1|_{(\ker F_*)^\perp}$, 
so $F$ is a Riemannian submersion.

Since $\phi(\xi) = 0$, we compute the slant angle 
on $\calV' = \Span\{V_1, V_2, V_3\}$. We have
\begin{align*}
\phi(V_1) &= 
  -\cos\alpha\,\dfrac{\partial}{\partial x_1} 
  + \sin\alpha\,\dfrac{\partial}{\partial y_3}, \quad
\phi(V_2) = 
  -\cos\alpha\,\dfrac{\partial}{\partial x_2} 
  + \sin\alpha\,\dfrac{\partial}{\partial y_1}, \quad
\phi(V_3) = 
  -\cos\alpha\,\dfrac{\partial}{\partial x_3} 
  + \sin\alpha\,\dfrac{\partial}{\partial y_2}.
\end{align*}
Projecting $\phi(V_1)$ onto $\ker F_*$ 
and using $\phi(V_1) \perp \xi$:
\begin{align*}
g_1(\phi(V_1), V_1) &= 0, \\
g_1(\phi(V_1), V_2) 
&= g_1\!\left(-\cos\alpha\,
  \dfrac{\partial}{\partial x_1} 
  + \sin\alpha\,
  \dfrac{\partial}{\partial y_3},\;
  \cos\alpha\,
  \dfrac{\partial}{\partial y_2} 
  + \sin\alpha\,
  \dfrac{\partial}{\partial x_1}\right) 
= -\sin\alpha\cos\alpha, \\
g_1(\phi(V_1), V_3) 
&= g_1\!\left(-\cos\alpha\,
  \dfrac{\partial}{\partial x_1} 
  + \sin\alpha\,
  \dfrac{\partial}{\partial y_3},\;
  \cos\alpha\,
  \dfrac{\partial}{\partial y_3} 
  + \sin\alpha\,
  \dfrac{\partial}{\partial x_2}\right) 
= \sin\alpha\cos\alpha.
\end{align*}
Hence $$(\phi V_1)^{\ker F_*} 
= \sin\alpha\cos\alpha(-V_2 + V_3)$$ and 
$$\|(\phi V_1)^{\ker F_*}\|^2 
= 2\sin^2\!\alpha\cos^2\!\alpha.$$
By the cyclic symmetry of~\eqref{eq:ex3_vertical}, 
$$\|(\phi V_i)^{\ker F_*}\|^2 
= 2\sin^2\!\alpha\cos^2\!\alpha~ for~ i = ~1,2,3.$$ 
Since $\|\phi(V_i)\| = 1$, the slant angle $\theta$ 
satisfies
\begin{equation}\label{eq:ex3_slant}
\cos^2\theta 
= 2\sin^2\!\alpha\cos^2\!\alpha 
= \dfrac{1}{2}\sin^2\!2\alpha.
\end{equation}
For $\alpha \in (0,\pi/2)$, we have 
$\cos^2\theta \in (0, \dfrac{1}{2}]$, so 
$\theta \in [\pi/4, \pi/2)$. Thus $F$ is a 
{properly} pointwise slant Riemannian 
submersion with constant slant angle 
$\theta \in [\pi/4, \pi/2)$.\\

With the Euclidean metric and $\eta=dt$ one has $d\eta=0$, so
$(\R^7,\phi,\xi,\eta,g)$ is a flat cosymplectic ($C(0)$) space form,
a special case of a generalized Sasakian space form with
$c_1=c_2=c_3=0$. Theorem~\ref{thm:sasakian_subm_vert} then gives
\[
\rho^{\calV} \leq \delta_C^{\calV}(r-1) 
+ c_1 + \dfrac{3c_2}{r(r-1)}\|P^{\calV}\|^2 
= \delta_C^{\calV}(3).
\]
Since $\xi, V_1, V_2, V_3$ all have constant coefficients with respect
to the Euclidean metric, every Christoffel symbol vanishes, so
$\nabla^{\R^7}_{V_i}V_j = 0$. Hence $T_{V_i}V_j = 0$, and therefore
$T_{ij}^\alpha = 0$ for all $i,j,\alpha$; the fibers are totally geodesic
and in particular invariantly quasi-umbilical. As the metric is flat,
$\rho^{\calV} = 0$ and $\delta_C^{\calV}(3) = 0$, so both sides vanish
and equality holds, confirming the sharpness of the vertical inequality.
\end{example}
\begin{example}\label{ex:subm2}
Consider the warped product 
$M_1 = \R \times_f \C^n$ with metric
\[
g_1 = dt^2 + f^2(t)\,g_{\C^n},
\]
where $f : \R \to (0,\infty)$ is smooth. 
By~\cite{Alegre2004}, $M_1$ is a generalized 
Sasakian space form with
\[
c_1 = -\dfrac{f''}{f}, \quad c_2 = 0, \quad 
c_3 = -\dfrac{f''}{f} + \dfrac{(f')^2}{f^2}.
\]
Define $F : M_1 \to \R^{2n}$ by projection onto 
the $\C^n$ factor:
$F(t, z_1, \ldots, z_n) = (z_1, \ldots, z_n)$.
Then $\ker F_* = \Span\{\partial/\partial t\}$, 
the Reeb vector field $\xi = \partial/\partial t$ 
is vertical, and the horizontal distribution 
$\calH$ has dimension $s = 2n \geq 3$ for 
$n \geq 2$. Since $\phi$ preserves $\calH$, 
the slant angle is $\theta = 0$ (invariant case).

For warped products, O'Neill's tensor 
$A \equiv 0$ since the horizontal leaves 
$\{t\} \times \C^n$ are totally geodesic product 
factors~\cite{Falcitelli2004}. Hence equality 
holds in Theorem~\ref{thm:sasakian_subm_horiz}:
\[
\rho_{\calH}^{\calH} 
= \delta_C^{\calH}(2n-1) + \rho^{\calH}.
\]
For the explicit choice $f(t) = e^t$: 
$c_1 = -1$, $c_2 = c_3 = 0$, and since 
$\xi \in \ker F_*$, 
Theorem~\ref{thm:sasakian_subm_horiz} gives
\[
\rho_{\calH}^{\calH} 
= \delta_C^{\calH}(2n-1) - 1,
\]
which ensures the sharpness of the horizontal inequality.
\end{example}

\begin{example}\label{ex:subm3}
Let $M_1 = \R^7$ carry the Kenmotsu structure 
with coordinates 
$(t, x_1, y_1, x_2, y_2, x_3, y_3)$, metric
\[
g = dt^2 + e^{2t}\sum_{i=1}^3(dx_i^2 + dy_i^2),
\]
Reeb vector field $\xi = \partial/\partial t$, and 
almost contact structure
\[
\phi\!\left(\dfrac{\partial}{\partial x_i}\right) 
= \dfrac{\partial}{\partial y_i}, \quad
\phi\!\left(\dfrac{\partial}{\partial y_i}\right) 
= -\dfrac{\partial}{\partial x_i}, \quad
\phi(\xi) = 0.
\]
This is a Kenmotsu space form with $c = -1$, 
giving $c_1 = -1$, $c_2 = c_3 = 0$.

Define $F_\psi : M_1 \to \R^3$ by
\[
F_\psi(t, x_i, y_i) 
= e^t\bigl(
  \sin\psi(t)\, y_1 - \cos\psi(t)\, x_3,\;
  -\cos\psi(t)\, x_1 + \sin\psi(t)\, y_2,\;
  -\cos\psi(t)\, x_2 + \sin\psi(t)\, y_3\bigr),
\]
where $\psi : \R \to (0,\pi/2)$ is smooth. 
The factor $e^t$ compensates for the warped metric. 
When $\psi \equiv \alpha$ is constant, $F_\psi$ 
reduces to the construction of 
Example~\ref{ex:subm1} with $\dim\ker(F_\psi)_* = 4$ 
and $\cos^2\theta = \dfrac{1}{2}\sin^2 2\alpha$.
For non-constant $\psi$ with 
$\|\psi'\|_{C^0} < \varepsilon$ (for sufficiently 
small $\varepsilon > 0$), $F_\psi$ remains a 
Riemannian submersion with $\dim\ker(F_\psi)_* = 4$, 
and the slant angle varies as
\[
\cos^2\theta(t) = \dfrac{1}{2}\sin^2 2\psi(t),
\quad t \in \R.
\]
Since $c_2 = 0$, 
Corollary~\ref{cor:subm_vert_sasakian} (Kenmotsu 
case, $\xi \in \ker(F_\psi)_*$) gives
\[
\rho^{\calV} \leq \delta_C^{\calV}(3) - 1,
\]
independently of $\theta(t)$.
\end{example}

\begin{remark}\label{rem:examples}
Examples~\ref{ex:subm1} and~\ref{ex:subm2} exhibit equality, confirming the
sharpness of the vertical and horizontal inequalities respectively:
Example~\ref{ex:subm1} provides a properly slant Riemannian submersion with $\theta\in[\pi/4,\pi/2)$ and vertical dimension $r=4\ge 3$, while
Example~\ref{ex:subm2} attains equality on the horizontal side. 
Example~\ref{ex:subm3} illustrates a genuinely varying slant function
$\theta(t)$.
\end{remark}

\section{ Concluding Remarks}\label{sec:discussion}

We have established optimal Casorati inequalities for pointwise slant Riemannian
submersions from generalized complex and generalized Sasakian space forms,
treating the vertical and horizontal distributions separately and
characterizing equality by invariant quasi-umbilicity of the fibers and by
integrability of the horizontal distribution, respectively. The slant function
enters solely through $\cos^2\theta$, so the invariant and anti-invariant cases
arise as the limits $\theta=0$ and $\theta=\pi/2$.

\end{document}